\documentclass[11pt]{amsart}

\usepackage{amsthm}
\usepackage{mathdots}
\usepackage{graphicx,color}
\usepackage{subfigure}

\newtheorem{thm0}{Theorem}[section]
\newtheorem{lemma}[thm0]{Lemma}
\newtheorem{prop}[thm0]{Proposition}
\newtheorem{rem}[thm0]{Remark}
\newtheorem{cor}[thm0]{Corollary}
\newtheorem{defin}[thm0]{Definition}
\newtheorem{exa}[thm0]{Example}

\begin{document}

\title{Logarithmic bundles of hypersurface arrangements in $ \mathbf{P}^n $}
\author{Elena Angelini}
\address{Dipartimento di Matematica e Informatica ``Ulisse Dini", Universit\`a di Firenze, Viale Morgagni 67/A, 50134 Firenze}
\email{elena.angelini@math.unifi.it}
\maketitle

%%%%%%%%%%%%%%%%%%%%%%%%%%%%%%%%%%%%%%%%%%%%%%%%%%%%%%%%%%%%%%%%%%%%%%%%%%%%%%%%%%%%%%%%%%%%%%%%%%%%%%%%%%%%%%%%%%%%%%%%%%%%%%%%%

\begin{abstract}
Let $ \mathcal{D}=\{D_{1}, \ldots, D_{\ell}\} $ be an arrangement of smooth hypersurfaces with normal crossings on the complex projective space $ \mathbf{P}^n $ and let $ \Omega^{1}_{\mathbf{P}^n}(\log \mathcal{D}) $ be the logarithmic bundle attached to it. Following \cite{AppuntiAncona}, we show that $ \Omega^{1}_{\mathbf{P}^n}(\log \mathcal{D}) $ admits a resolution of lenght $ 1 $ which explicitly depends on the degrees and on the equations of $ D_{1}, \ldots, D_{\ell} $. Then we prove a Torelli type theorem when all the $ D_{i} $'s have the same degree $ d $ and $ \ell \geq {{n+d} \choose {d}}+3 $: indeed, we recover the components of $ \mathcal{D} $ as unstable smooth hypersurfaces of $ \Omega^{1}_{\mathbf{P}^n}(\log \mathcal{D}) $. Finally we analyze the cases of one quadric and a pair of quadrics, which yield examples of non-Torelli arrangements. In particular, through a duality argument, we prove that two pairs of quadrics have isomorphic logarithmic bundles if and only if they have the same tangent hyperplanes.  \\



\end{abstract}

\tableofcontents

\section{Introduction}

Let $ X $ be a non singular algebraic variety of dimension $ n $ and let $ \mathcal{D} $ be a union of $ \ell $ distinct smooth irreducible hypersurfaces on $ X $, which we call an \emph{arrangement} on $ X $.  We can associate to $ \mathcal{D} $ the \emph{sheaf of differential $ 1 $-forms with logarithmic poles on $ \mathcal{D} $}, denoted by $ \Omega^{1}_{X}(\log \mathcal{D}) $. This sheaf was originally introduced by Deligne in~\cite{De} for an arrangement with normal crossings. In this case, for all $ x \in X $, the space of sections of $ \Omega^{1}_{X}(\log \mathcal{D}) $ near $ x $ is defined by 
$$ <d\log z_{1}, \ldots, d\log z_{k}, dz_{k+1}, \ldots, dz_{n}>_{\mathcal{O}_{X,x}} $$
where $ z_{1}, \ldots, z_{n} $ are local coordinates such that $ \mathcal{D} = \{z_{1} \cdot \ldots \cdot z_{k} = 0\} $. In particular, $ \Omega^{1}_{X}(\log \mathcal{D}) $ turns out to be a vector bundle over $ X $ and it is called \emph{logarithmic bundle}.\\
Once we construct the correspondence
\begin{equation}\label{eq:Tmap}
\mathcal{D} \longrightarrow \Omega^{1}_{X}(\log \mathcal{D})
\end{equation}
a natural, interesting question is whether $ \Omega^{1}_{X}(\log \mathcal{D}) $ contains information enough to recover $ \mathcal{D} $. Since the injectivity of the map in (\ref{eq:Tmap}) is investigated, we can talk about the \emph{Torelli problem} for $ \Omega^{1}_{X}(\log \mathcal{D}) $. In particular, if the isomorphism class of $ \Omega^{1}_{X}(\log \mathcal{D}) $ determines $ \mathcal{D} $, then $ \mathcal{D} $ is called a \emph{Torelli arrangement}. \newline
In the mathematical literature, the first situation that has been analyzed is the case of hyperplanes in the complex projective space $ \mathbf{P}^{n} $. Hyperplane arrangements represent an important topic in geometry, topology and combinatorics (\cite{OT},~\cite{Arn}). In 1993 Dolgachev and Kapranov gave an answer to the Torelli problem when $ \mathcal{H} = \{H_{1}, \ldots, H_{\ell}\} $ is an arrangement of hyperplanes with normal crossings,~\cite{Do-Ka}. They proved that if $ \ell \leq n+2 $ then two different arrangements give always the same logarithmic bundle; moreover, if $ \ell \geq 2n+3 $, then we can reconstruct $ \mathcal{H} $ from $ \Omega^{1}_{\mathbf{P}^n}(\log \mathcal{H}) $ unless the hyperplanes in $ \mathcal{H} $ don't osculate a rational normal curve $ \mathcal{C}_{n} $ of degree $ n $ in $ \mathbf{P}^n $, in which case $ \Omega^{1}_{\mathbf{P}^n}(\log \mathcal{H}) $ is isomorphic to $ E_{\ell - 2}(\mathcal{C}_{n}^{\vee}) $, the \emph{Schwarzenberger bundle} (\cite{Sch2}, \cite{Sch1}) \emph{of degree $ \ell-2 $ associated to $ \mathcal{C}_{n}^{\vee} $}. In 2000 Vall\`es extended the latter result to $ \ell \geq n+3 $,~\cite{V}: while Dolgachev and Kapranov studied the set of \emph{jumping lines} (\cite{Ba},~\cite{Hu}) of $ \Omega^{1}_{\mathbf{P}^n}(\log \mathcal{H}) $, Vall\`es characterized $ \mathcal{H} $ as the set of \emph{unstable hyperplanes} of the logarithmic bundle, i.e. $ \{ H \subset \mathbf{P}^{n} \, hyperplane \, | \, H^{0}(H, \Omega^{1}_{\mathbf{P}^n}(\log \mathcal{H})_{|_{H}}^{\vee}) \not = \{0\} \}. $ \\
Concerning the higher degree case, Ueda and Yoshinaga proved that an arrangement consisting of a smooth plane cubic is Torelli if and only if the cubic has non-vanishing j-invariant (\cite{UY1}). In (\cite{UY2}) they extended the previous result to the case of a smooth hypersurface in $ \mathbf{P}^n $, showing that such arrangement is Torelli if and only if its defining equation is not of \emph{Sebastiani-Thom type} (Theorem \ref{T:UY2}). \\
In this paper, after recalling the fundamental definitions and the main classical results on the subject, we consider arrangements of higher degree smooth hypersurfaces with normal crossings on $ \mathbf{P}^n $. Following \cite{AppuntiAncona}, in Theorem~\ref{T:Ancona} we prove that $ \Omega^{1}_{\mathbf{P}^n}(\log \mathcal{D}) $ admits a resolution of lenght $ 1 $ which is a very important tool for our investigations. In particular, this resolution allows us to find again the aforementioned results concerning hyperplanes. Our main results are collected in sections $ 5, \,6$ and $ 7 $. Section $ 5 $ is devoted to arrangements made of a sufficiently large number of hypersurfaces of the same degree $ d $ in $ \mathbf{P}^n $: if $ \ell \geq {{n+d} \choose {d}}+3 $, then we can recover the components of $ \mathcal{D} $ as \emph{unstable hypersurfaces} of $ \Omega^{1}_{\mathbf{P}^n}(\log \mathcal{D}) $, unless the hyperplanes in $ \mathbf{P}^{{{n+d} \choose {d}}-1} $ corresponding to $ D_{1}, \ldots, D_{\ell} $ through the \emph{Veronese map of degree} $ d $ satisfy further hypothesis (Theorem~\ref{t54}). The notion of unstable hypersurface is inspired on the one of unstable hyperplane recalled previously. In section $ 6 $ and $ 7 $ we study the cases of one quadric and two quadrics (Proposition~\ref{p:1quadrica} and Theorem~\ref{T:pairquad}) and we find arrangements which are not of Torelli type. In particular, in the second case, by using the \emph{simultaneous diagonalization} and a \emph{duality} argument, we prove that two pairs of quadrics are associated to isomorphic logarithmic bundles if and only if they have the same tangent hyperplanes. \\

\noindent {\bf Acknowledgements} This paper collects some results of my Ph.D. thesis, mainly prepared at the University of Florence, in collaboration with Universit\'e de Pau et des pays de l'Adour (Pau). I would like to thank my advisors Professor Giorgio Ottaviani and Professor Daniele Faenzi for suggesting me the subject, for many fruitful discussions and their help. A special thank goes to Professor Vincenzo Ancona, since the Theorem \ref{T:Ancona} is based on the notes of a talk that he gave in 1998 and that he allowed me to examine.

\section{Preliminar definitions and notations}

We suppose that everything is defined over $ \mathbf{C} $. Let $ X $ be a smooth algebraic variety, we give the following:

\begin{defin}
An \emph{arrangement} on $ X $ is a family $ \mathcal{D} = \{D_{1}, \ldots, D_{\ell}\} $ of smooth irreducible hypersurfaces of $ X $ such that $ D_{i} \not= D_{j} $ for all $ i,j \in \{1, \ldots, \ell\} $, $ i \not= j $. We say that such $ \mathcal{D} $ has \emph{normal crossings} if it is locally isomorphic (in the sense of holomorphic local coordinates changes) to a union of coordinate hyperplanes of $ \mathbf{C}^n $.
\end{defin}

\begin{exa}
Let $ \mathcal{D} $ be an arrangement on the $ n $-dimensional complex projective space, which we simply denote by $ \mathbf{P}^n $. Each hypersurface $ D_{i} \in \mathcal{D} $ is defined as the zero locus of a homogeneous polynomial $ f_{i} $ of degree $ d_{i} $ in the variables $ x_{0}, \ldots, x_{n} $. Thus $ \mathcal{D} $ is given by the set of zeroes of $ f_{1} \cdot \ldots \cdot f_{\ell} $, which is a polynomial of degree $ d_{1} + \ldots + d_{\ell} $. \\
A \emph{hyperplane arrangement} $ \mathcal{H} = \{H_{1}, \ldots, H_{\ell}\} $ is a hypersurface arrangement with $ d_{i} = 1 $ for all $ i $. We have that $ \mathcal{H} $  has normal crossings if and only if $ codim(H_{i_{1}} \cap \ldots \cap H_{i_{k}}) = k $ for any $ k \leq n+1 $ and $ 1 \leq i_{1} < \ldots < i_{k} \leq \ell $. \\
If all $ d_{i} $'s are equal to $ 2 $ we have an \emph{arrangement of quadrics}. Theorem \ref{T:coppiequadriche} yields a description of the normal crossings condition in the case of a pair of quadrics. 
\end{exa}

Let $ \mathcal{D} $ be an arrangement with normal crossings on $ X $. In order to introduce the notion of \emph{sheaf of logarithmic forms} on $ \mathcal{D} $ we will refer to Deligne (\cite{De0},~\cite{De}). This is not the unique way to describe these sheaves, there are also other definitions for more general divisors (\cite{Sa},~\cite{Schenck}) that are equivalent to this one for arrangements with normal crossings. 

Let $ U = X - \mathcal{D} $ be the complement of $ \mathcal{D} $ in $ X $ and let $ j: U \hookrightarrow X $ be the embedding of $ U $ in $ X $. We denote by $ \Omega_{U}^1$ the sheaf of holomorphic differential $ 1 $-forms on $ U $ and by $ j_{\ast}\Omega_{U}^1 $ its direct image sheaf on $ X $. We remark that, since $ \mathcal{D} $ has normal crossings, then for all $ x \in X $ there exists a neighbourhood $ I_{x} \subset X $ such that $ I_{x} \cap \mathcal{D} = \{z_{1} \cdots z_{k} = 0 \} $, where $ \{z_{1}, \ldots, z_{k}\} $ is a part of a system of local coordinates. We have the following:

\begin{defin}
We call \emph{sheaf of differential $ 1 $-forms on $ X $ with logarithmic poles on $ \mathcal{D} $} the subsheaf $ \Omega_{X}^{1}(\log \mathcal{D}) $ of $ j_{\ast}\Omega_{U}^1 $, such that, for all $ x \in X $,
$ \Gamma(I_{x},\Omega_{X}^{1}(\log \mathcal{D})) $ is given by
$$ \{s \in \Gamma(I_{x},j_{\ast}\Omega_{U}^1) \,|\, s = \displaystyle{\sum_{i=1}^k u_{i} d\log z_{i}} \, +  \displaystyle{\sum_{i=k+1}^n v_{i} dz_{i}} \} $$
where $ u_{i}, v_{i} $ are locally holomorphic functions and $ d \log z_{i} = \displaystyle{{dz_{i}} \over {z_{i}}} $.
\end{defin}

\begin{rem}
Since $ \mathcal{D} $ has normal crossings, $ \Omega_{X}^{1}(\log \mathcal{D}) $ is a locally free sheaf of rank $ n = dim X $,~\cite{De}. So, $ \Omega_{X}^{1}(\log \mathcal{D}) $ can be regarded as a rank-$ n $ vector bundle on $ X $ and it is called the \emph{logarithmic bundle attached to $ \mathcal{D} $}. In particular, $ \Omega_{X}^{1}(\log \mathcal{D}) $ admits the \emph{residue exact sequence}
\begin{equation}\label{eq:resexseq}
0 \longrightarrow \Omega_{X}^1 \longrightarrow \Omega_{X}^{1}(\log \mathcal{D}) \buildrel \rm res \over \longrightarrow \, \displaystyle{\bigoplus_{i=1}^{\ell} \mathcal{O}_{D_{i}}} \longrightarrow 0
\end{equation}
where \emph{res} denotes the \emph{Poincar\'e residue morphism}, \cite{Do}.
\end{rem}

Given a smooth algebraic variety $ X $, we are able to map an arrangement with normal crossings on $ X $ to a logarithmic vector bundle on $ X $:
\begin{equation}\label{eq:Torellimap}
\mathcal{D} \longmapsto \Omega_{X}^{1}(\log \mathcal{D}). 
\end{equation}
A natural question arises from this contruction: is it true that isomorphic logarithmic bundles come from the same arrangement? If the answer is positive, then we say that $ \mathcal{D} $ is an \emph{arrangement of Torelli type}, or a \emph{Torelli arrangement}. This is the so called \emph{Torelli problem for logarithmic bundles}. \\
In the next section we will discuss the main results concerning arrangements of hyperplanes with normal crossings in the complex projective space, mainly referring to Dolgachev and Kapranov (\cite{Do-Ka}), Ancona and Ottaviani (\cite{AO}), Vall\`es (\cite{V}).

\section{Some known results about hyperplane arrangements}

Let $ \mathcal{H} = \{H_{1}, \ldots, H_{\ell}\} $ be a hyperplane arrangement with normal crossings on $ \mathbf{P}^n $ and let $ \Omega_{\mathbf{P}^n}^{1}(\log \mathcal{H}) $ the corresponding logarithmic bundle. The description of the \emph{Torelli problem} in this case depends on $ \ell $. In this sense, if $ \mathcal{H} $ is made of \emph{few} hyperplanes, then it is not of \emph{Torelli type}:

\begin{thm0}\label{T:pochiiperpiani}$($\emph{Dolgachev-Kapranov 1993,~\cite{Do-Ka}}$)$\\
If $ 1 \leq \ell \leq n+1 $ then $ \Omega_{\mathbf{P}^n}^{1}(\log \mathcal{H}) \cong \mathcal{O}_{\mathbf{P}^n}^{\ell-1} \oplus \mathcal{O}_{\mathbf{P}^n}(-1)^{n+1-\ell}.  $
\end{thm0}

If $ \ell \geq n+2 $, Dolgachev and Kapranov (\cite{Do-Ka}) proved that $ \Omega_{\mathbf{P}^n}^{1}(\log \mathcal{H}) $ belongs to the family $ S_{n, \ell-n-1} $ of \emph{Steiner bundles} with parameters $ n $ and $ \ell-n-1 $, that is it admits the short exact sequence
\begin{equation}\label{eq:steinerfacile}
0 \longrightarrow \mathcal{O}_{\mathbf{P}^n}(-1)^{\ell-n-1} \longrightarrow \mathcal{O}_{\mathbf{P}^n}^{\ell-1} \longrightarrow \Omega_{\mathbf{P}^n}^{1}(\log \mathcal{H}) \longrightarrow 0.
\end{equation}
In particular $ \Omega_{\mathbf{P}^n}^{1}(\log \mathcal{H}) $  is stable in the sense of \emph{Mumford-Takemoto}. So we immediately get the following:

\begin{prop}\label{p:n+2iperpiani}
If $ \ell = n+2 $ then $ \Omega_{\mathbf{P}^n}^{1}(\log \mathcal{H}) \cong \mathbf{TP}^{n}(-1). $
\end{prop}

If $ \ell $ is \emph{sufficiently large}, then the Torelli correspondence defined in (\ref{eq:Torellimap}) is very close to an injective map. In order to state the main result in this direction, we recall that, given $ m \in \mathbf{N} $ and a rational normal curve $ \mathcal{C}_{n}^{\vee} \subset (\mathbf{P}^n)^{\vee} $ of degree $ n $, the \emph{Schwarzenberger bundle of degree} $ m $ \emph{associated to} $ \mathcal{C}_{n}^{\vee} $ is the rank-$ n $ vector bundle over $ \mathbf{P}^n $ given by
$$ E_{m}(\mathcal{C}_{n}^{\vee}) = \overline{p}_{\ast}\overline{q}^{\ast} \mathcal{O}_{\mathcal{C}_{n}^{\vee}}\left({{m} \over {n}}\right) $$
where $ \mathcal{O}_{\mathcal{C}_{n}^{\vee}}({{m} \over {n}}) $ denotes the line bundle over $ \mathcal{C}_{n}^{\vee} $ that corresponds to $ \mathcal{O}_{\mathbf{P}^{1}}(m) $ through the chosen isomorphism between $ \mathbf{P}^{1} $ and $ \mathcal{C}_{n}^{\vee} $. In the previous equality $\overline{p} $ and $ \overline{q} $ are the restrictions to $ q^{-1}(\mathcal{C}_{n}^{\vee}) \subset \mathbf{F} $ of $ p $ and $ q $, the canonical projection morphisms from the incidence variety point-hyperplane $ \mathbf{F}$ to the factors $ \mathbf{P}^{n} $ and $  (\mathbf{P}^n)^{\vee} $ (for more detailed descriptions see \cite{Sch2} and \cite{Sch1}). In particular, Dolgachev and Kapranov (\cite{Do-Ka}) proved that, if $ m \geq n $, then $ E_{m}(\mathcal{C}_{n}^{\vee}) $ is a logarithmic bundle of an arrangement with $ \ell = m+2 $ hyperplanes with normal crossings osculating $ \mathcal{C}_{n} \subset \mathbf{P}^n $. \\
We have the following:

\begin{thm0}\label{T:V}$($\emph{Vall\`es 2000,~\cite{V}}$)$\\
Let $ \mathcal{H} = \{H_{1}, \ldots, H_{\ell}\} $ and $ \mathcal{K} = \{K_{1}, \ldots, K_{\ell}\} $ be arrangements of $ \ell \geq n+3 $ hyperplanes with normal crossings on $ \mathbf{P}^n $ such that 
\begin{equation}\label{eq:iperlogiso}
\Omega_{\mathbf{P}^n}^{1}(\log \mathcal{H}) \cong \Omega_{\mathbf{P}^n}^{1}(\log \mathcal{K}).
\end{equation}
Then one of the following two cases occurs:
\begin{itemize}
\item[$1)$] $ \mathcal{H} = \mathcal{K} $;
\item[$2)$] there exists a rational normal curve $ \mathcal{C}_{n} \subset \mathbf{P}^n $ such that $ H_{1}, \ldots, H_{\ell}, \\
K_{1}, \ldots, K_{\ell} $ osculate $ \mathcal{C}_{n} $ and $ \Omega_{\mathbf{P}^n}^{1}(\log \mathcal{H}) \cong \Omega_{\mathbf{P}^n}^{1}(\log \mathcal{K}) \cong E_{\ell - 2}(\mathcal{C}_{n}^{\vee}) $. 
\end{itemize}
\end{thm0}

In 1993 Dolgachev and Kapranov proved Theorem \ref{T:V} when $ \ell \geq 2n+3 $, focusing their attention on the set of \emph{jumping lines} of $ \Omega_{\mathbf{P}^n}^{1}(\log \mathcal{H}) $, \cite{Do-Ka}. On the contrary, Vall\`es' proof is based on the following idea: recover the hyperplanes of $ \mathcal{H} $ as \emph{unstable hyperplanes} of $ \Omega_{\mathbf{P}^n}^{1}(\log \mathcal{H}) $. We recall that, a hyperplane $ H $ in $ \mathbf{P}^n $ is \emph{unstable} for $ \Omega_{\mathbf{P}^n}^{1}(\log \mathcal{H}) $ if 
\begin{equation}\label{eq:unsthyp}
H^{0}(H, \Omega_{\mathbf{P}^n}^{1}(\log \mathcal{H})^{\vee}_{|_{H}}) \not= \{0\}. 
\end{equation} 
Theorem \ref{T:V} is a consequence of the following result:

\begin{thm0}\label{T:V0}$($\emph{Vall\`es 2000,~\cite{V}}$)$\\
Let $ \mathcal{H} = \{H_{1}, \ldots, H_{\ell}\} $ be an arrangement of $ \ell \geq n+3 $ hyperplanes with normal crossings on $ \mathbf{P}^n $. Then 
$$ \mathcal{H} = \{ H \subset \mathbf{P}^{n} \, hyperplane \, | \, H^{0}(H, \Omega^{1}_{\mathbf{P}^n}(\log \mathcal{H})_{|_{H}}^{\vee}) \not = \{0\} \}, $$
unless $ H_{1}, \ldots, H_{\ell} $ osculate a rational normal curve $ \mathcal{C}_{n} $ of degree $ n $ in $ \mathbf{P}^{n} $, in which case all the hyperplanes corresponding to the points of $ \mathcal{C}_{n}^{\vee} \subset (\mathbf{P}^{n})^{\vee} $ satisfy (\ref{eq:unsthyp}) and $ \Omega_{\mathbf{P}^n}^{1}(\log \mathcal{H}) \cong E_{\ell - 2}(\mathcal{C}_{n}^{\vee}) $.
\end{thm0}

Other interesting results concerning Steiner bundles and unstable hyperplanes have been proved by Ancona and Ottaviani in~\cite{AO}. \\
A few years ago, Dolgachev in \cite{Do} and Faenzi-Matei-Vall\`es in~\cite{FMV} investigated hyperplane arrangements without normal crossings and studied the Torelli problem for the subsheaf $ \widetilde{\Omega}^{1}_{\mathbf{P}^n}(\log \mathcal{H}) $ of $ \Omega^{1}_{\mathbf{P}^n}(\log \mathcal{H}) $ introduced in \cite{CHKS}. In particular, in \cite{FMV} is proved that $ \mathcal{H} $ is a Torelli arrangement if and only if $ H_{1}, \ldots, H_{\ell} $, seen as point in the dual projective space $ ({\mathbf{P}^n})^{\vee} $, belong to a \emph{Kronecker-Weierstrass variety of type $ (d, s) $}, which is essentially the union of a smooth rational curve of degree $ d $ with s linear subspaces. \\

\section{The higher degree case}

According to the previous section, the \emph{Torelli problem} for hyperplane arrangements  has been completely solved, but in the higher degree case it still represents an open question. \\
A first step towards this direction is due to Ueda and Yoshinaga. In~\cite{UY1} they studied this problem for one smooth cubic $ D $ in $ \mathbf{P}^2 $, focusing their attention on the set of jumping lines of the corresponding logarithmic bundle. In this sense they proved that, if $ D $ has non-vanishing $ j $-invariant, then the map in (\ref{eq:Torellimap}) is injective. Afterwards, in~\cite{UY2} they extended the previous result to the case of one smooth hypersurface in $ \mathbf{P}^n $. In particular they proved the following:

\begin{thm0}\label{T:UY2}$($\emph{Ueda-Yoshinaga 2009,~\cite{UY2}}$)$ \\
Let $ D = \{f = 0\} $ be a smooth hypersurface of degree $ d $ in $ \mathbf{P}^n $. $ \mathcal{D} = \{D\} $ is a \emph{Torelli arrangement} if and only if $ f $ is not of \emph{Sebastiani-Thom type}, that is we can't choose homogeneous coordinates $ x_{0}, \ldots, x_{n} $ of $ \mathbf{P}^n $ and a number $ k \in \{0, \ldots, n-1\} $ such that $ f(x_{0}, \ldots, x_{n}) = f_{1}(x_{0}, \ldots, x_{k}) + f_{2}(x_{k+1}, \ldots, x_{n}). $
\end{thm0}

We remark that, if $ d = 2 $, then $ f $ is always of Sebastiani-Thom type. Moreover, in the case of $ d = 3 $, the $ j $-invariant vanishes if and only if there is a choice of coordinates such that $ D $ is the zero locus of the \emph{Fermat polynomial} $ x_{0}^3+x_{1}^{3}+x_{2}^{3} $ which is equivalent to say that $ D $ is defined by a polynomial of Sebastiani-Thom type. We can state the following:

\begin{cor}
Let $ D $ be a general hypersurface of degree $ d $ in $ \mathbf{P}^n $. Then $ \mathcal{D} = \{D\} $ is \emph{Torelli} if and only if $ d \geq 3 $.
\end{cor}

At the best of my knowledge, in the mathematical literature there aren't descriptions concerning the the higher degree case with $ \ell \geq 2 $. In this sense, let $ \mathcal{D} = \{D_{1}, \ldots, D_{\ell}\} $ be an arrangement of smooth hypersurfaces with normal crossings on $ \mathbf{P}^n $ and let $ \Omega_{\mathbf{P}^n}^{1}(\log \mathcal{D}) $ be the associated logarithmic bundle. Assume that, for all $ i \in \{1, \ldots, \ell\} $, $ D_{i} = \{f_{i} = 0\} $, where $ f_{i} $ is a homogeneous polynomial of degree $ d_{i} $ in $ x_{0}, \ldots, x_{n} $. We denote by $ \partial_{j}f_{i} $ the partial derivative of $ f_{i} $ with respect to $ x_{j} $. \\
Our investigations are based on the following:

\begin{thm0}\label{T:Ancona}$($\emph{Ancona,~\cite{AppuntiAncona}}$)$\\
The dual bundle of $ \Omega_{\mathbf{P}^n}^{1}(\log \mathcal{D}) $ admits the short exact sequence
\begin{equation}\label{eq:Anconas}
0 \longrightarrow \Omega_{\mathbf{P}^n}^{1}(\log \mathcal{D})^{\vee} \longrightarrow \mathcal{O}_{\mathbf{P}^n}(1)^{n+1} \oplus \mathcal{O}_{\mathbf{P}^n}^{\ell-1} \buildrel \rm N \over \longrightarrow \displaystyle{\bigoplus_{i=1}^{\ell} \mathcal{O}_{\mathbf{P}^n}(d_{i})} \longrightarrow 0
\end{equation}
where $ N $ is the $ \ell \times (n + \ell) $ matrix 
\begin{equation}\label{eq:matrixAncona}
N = \begin{pmatrix} 
\partial_{0}f_{1} & \cdots & \partial_{n}f_{1} & f_{1} & 0 & \cdots & 0 \cr 
\partial_{0}f_{2} & \cdots & \partial_{n}f_{2} & 0 & f_{2} & {} & \vdots \cr
\vdots & {} & \vdots & \vdots & {} & \ddots & 0 \cr 
\partial_{0}f_{\ell-1} & \cdots & \partial_{n}f_{\ell-1} & 0 & \cdots & 0 & f_{\ell-1} \cr
\partial_{0}f_{\ell} & \cdots & \partial_{n}f_{\ell} & 0 & \cdots & \cdots & 0 \cr
\end{pmatrix}. 
\end{equation}
\end{thm0}

\begin{proof}
As in \cite{Do}, let us denote by $ S $ the polynomial algebra $ \mathbf{C}[x_{0}, \ldots, x_{n}] $ and let 
$$ \Omega^{1}_{S} = <dx_{0}, \ldots, dx_{n}>_{S} \cong S(-1)^{n+1} $$
$$ Der_{S} = <{{\partial} \over {\partial x_{0}}}, \ldots, {{\partial} \over {\partial x_{n}}}>_{S} \cong S(1)^{n+1} $$
be, respectively, the graded $ S $-module of differentials and the graded $ S $-module of derivations. The Euler derivation $ \xi = \displaystyle{\sum_{i=0}^{n} x_{i} {{\partial} \over {\partial x_{i}}}} $ defines a homomorphism of graded $ S $-modules 
$$ \Omega^{1}_{S} \longrightarrow S \quad\,\,\,\, \,\, \omega = \displaystyle{\sum_{i=0}^{n} h_{i}dx_{i}}  \longmapsto \omega(\xi) = \displaystyle{\sum_{i=0}^{n} h_{i}x_{i}} $$
whose kernel corresponds to the sheaf $ \Omega^{1}_{\mathbf{P}^{n}} $. Moreover the cokernel of the homomorphism
$$ S \longrightarrow Der_{S} \quad\,\,\,\, p \longmapsto p \xi $$
corresponds to $ \mathbf{TP}^{n} $, which is the dual sheaf of $ \Omega^{1}_{\mathbf{P}^{n}} $.
So we have a pairing
$$ \Omega^{1}_{\mathbf{P}^{n}} \times \mathbf{TP}^{n} \buildrel \rm < \cdot, \cdot > \over \longrightarrow \mathcal{O}_{\mathbf{P}^{n}} $$
$$ \left(\displaystyle{\sum_{i=0}^{n} h_{i}dx_{i}}, \displaystyle{\sum_{i=0}^{n} b_{i}{{\partial} \over {\partial x_{i}}}}\right) \longmapsto \displaystyle{\sum_{i=0}^{n} h_{i}b_{i}} $$
and, if $ U $ is an open subset of $ \mathbf{P}^{n} $, then $ \Gamma(U, \Omega_{\mathbf{P}^n}^{1}(\log \mathcal{D})^{\vee}) $ is given by
$$  \{ v \in \Gamma(U, \mathbf{TP}^{n}) \, | \, \forall\,local\,equation\, g_{i}\,of\,D_{i}\,in\,U <d\log g_{i},v> \, is \, holomorphic \} $$
where we recall that $ <d\log g_{i},v> = <\displaystyle{{d g_{i}} \over {g_{i}}}, v> $.  \\
Assume that $ x_{0} \not = 0 $ and let $ z_{j} = \displaystyle{{x_{j}} \over {x_{0}}} $ for all $ j \in \{1, \ldots, n\} $.
Since for all $ i \in \{1, \ldots, \ell\} $ we have that $ f_{i}(x_{0}, \ldots, x_{n}) = x_{0}^{d_{i}}f_{i}(1, \displaystyle{{x_{1}} \over {x_{0}}}, \ldots, \displaystyle{{x_{n}} \over {x_{0}}}) = x_{0}^{d_{i}} g_{i}(z_{1}, \ldots, z_{n}) $, the chain rule implies that, for all $ j \in \{1, \ldots, n\} $,
$$ \displaystyle{{\partial g_{i}} \over {\partial z_{j}}} = \displaystyle{{1} \over {x_{0}^{d_{i}-1}}} {{\partial f_{i}} \over {\partial x_{j}}}. $$
So we get that
$$ dg_{i} = \displaystyle{\sum_{j=1}^{n} {{\partial g_{i}} \over {\partial z_{j}}} dz_{j}} = \displaystyle{\sum_{j=1}^{n}} {{1} \over {x_{0}^{d_{i}-1}}} {{\partial f_{i}} \over {\partial x_{j}}} \left( \displaystyle{{dx_{j}} \over {x_{0}}} - \displaystyle{{x_{j}} \over {x_{0}^{2}}} dx_{0} \right) =   $$
$$ = \displaystyle{{1} \over {x_{0}^{d_{i}}}} \displaystyle{\sum_{j=1}^{n} {{\partial f_{i}} \over {\partial x_{j}}} dx_{j} } - \displaystyle{{dx_{0}} \over {x_{0}^{d_{i}+1}}} \displaystyle{\sum_{j=1}^{n}} {{\partial f_{i}} \over {\partial x_{j}}} x_{j} = \displaystyle{{1} \over {x_{0}^{d_{i}}}} \displaystyle{\sum_{j=1}^{n} {{\partial f_{i}} \over {\partial x_{j}}} dx_{j} } - \displaystyle{{dx_{0}} \over {x_{0}^{d_{i}+1}}} (d_{i})(f_{i}). $$
Thus we have that $ v = \displaystyle{\sum_{j=0}^{n} b_{j}{{\partial} \over {\partial x_{j}}}} \in \Gamma(U, \Omega_{\mathbf{P}^n}^{1}(\log \mathcal{D})^{\vee}) $ if and only if for all $ i \in \{1, \ldots, \ell\} $ there exists a holomorphic function $ \alpha_{i} $ such that 
$$ \displaystyle{\sum_{j=1}^{n} {{\partial f_{i}} \over {\partial x_{j}}} b_{j}} = \alpha_{i} f_{i} \quad\quad modulo \,\,\, \xi. $$
$ \Omega_{\mathbf{P}^n}^{1}(\log \mathcal{D})^{\vee} $ is the cohomology of the monad given by
$$ 0 \longrightarrow \mathcal{O}_{\mathbf{P}^{n}} \buildrel \rm \mathcal{M} \over \longrightarrow \mathcal{O}_{\mathbf{P}^{n}}(1)^{n+1} \oplus \mathcal{O}_{\mathbf{P}^{n}}^{\ell} \buildrel \rm \mathcal{N} \over \longrightarrow \displaystyle{\bigoplus_{i=1}^{\ell}\mathcal{O}_{\mathbf{P}^{n}}(d_{i})} \longrightarrow 0 $$
where $ \mathcal{M} $ is the $ (n+1+\ell) \times 1 $ matrix
$$ \mathcal{M} = \,\displaystyle{^{t}}\begin{pmatrix} x_{0} & \cdots & x_{n} & d_{1} & \cdots & d_{\ell} \cr \end{pmatrix} $$
and $ \mathcal{N} $ is the $ \ell \times (n+1+\ell) $ matrix
$$ \mathcal{N} = \begin{pmatrix} 
\partial_{0}f_{1} & \cdots & \partial_{n}f_{1} & f_{1} & 0 & \cdots & \cdots & 0 \cr 
\partial_{0}f_{2} & \cdots & \partial_{n}f_{2} & 0 & f_{2} & 0 & \cdots & 0 \cr
\vdots & {} & \vdots & \vdots & {} & \ddots & {} & \vdots \cr 
\vdots & {} & \vdots & \vdots & {} & {} &  \ddots & 0 \cr
\partial_{0}f_{\ell} & \cdots & \partial_{n}f_{\ell} & 0 & \cdots & \cdots & 0 & f_{\ell} \cr
\end{pmatrix}. $$
We remark that if we multiply $ \mathcal{N} $ with the square matrix of order $ n+\ell+1 $ 
$$ \begin{pmatrix} 
{} & {} & {} & {} & {} & x_{0} \cr
{} & {} & {} & {} & {} & \vdots \cr
{} & {} & {} & {}  & {} & x_{n} \cr
{} & {} & I_{n+\ell} & {} & {} & -d_{1} \cr
{} & {} & {} & {} & {} & \vdots \cr
{} & {} & {} & {} & {} & -d_{\ell - 1} \cr
0 & {} & \cdots & {} & 0 & -d_{\ell} \cr
\end{pmatrix} $$
and we apply the Euler formula, then we can remove the last column of $ \mathcal{N} $ so that $ \mathcal{N} $ takes the form of $ N $, the matrix in (\ref{eq:matrixAncona}). So $ \Omega_{\mathbf{P}^n}^{1}(\log \mathcal{D})^{\vee} $ admits the short exact sequence (\ref{eq:Anconas}), as desired.   
\end{proof}

\begin{rem}
At the best of my knowledge, the proof of the above theorem doesn't appear in the mathematical literature, even if a sequence similar to (\ref{eq:Anconas}) has been used in \cite{HKS} (see step 1 in algorithm 18 of section 7).
\end{rem}

\begin{rem}
Theorem~\ref{T:Ancona} holds in particular when we consider a hyperplane arrangement, that is when $ d_{i} = 1 $ for all $ i $. Indeed, if $ \ell \geq n+2 $ then (\ref{eq:Anconas}) becomes the dualized sequence of the \emph{Steiner sequence} (\ref{eq:steinerfacile}) and if $ \ell \leq n+1 $ then (\ref{eq:Anconas}) implies Theorem~\ref{T:pochiiperpiani}.   
\end{rem}

\section{Many higher degree hypersurfaces}

Arrangements consisting of a \emph{sufficiently large} number of hypersurfaces with normal crossings can be studied by using the main results concerning hyperplane arrangements with normal crossings (\cite{Do-Ka}, \cite{V}) that are recalled in section $ 3 $.

Let $ \mathcal{D} = \{D_{1}, \ldots, D_{\ell}\} $ be an arrangement of $ \ell $ smooth hypersurfaces of the same degree $ d \geq 2 $ with normal crossings on $ \mathbf{P}^n $, $ n \geq 2 $. We denote by $ \Omega_{\mathbf{P}^n}^{1}(\log \mathcal{D}) $ the corresponding logarithmic bundle. If $ n=2 $ each $ D_{i} $ is a curve, so we deal with arrangements of conics, cubics, and so on.

\begin{rem}\label{r:Veronese}
Let us consider the Veronese map of degree $ d $, that is 
$$ \nu_{d} : \mathbf{P}^n \longrightarrow \mathbf{P}^{N} $$
$$ \quad [x_{0}, \ldots, x_{n}] \longmapsto [\ldots \, x^{I} \ldots] $$
where $ N = {{n+d} \choose {d}} - 1 $ and $ x^{I} $ ranges over all monomials of degree $ d $ in $ x_{0},\ldots, x_{n} $. Let $ V_{d} = \nu_{d}(\mathbf{P}^n) $ be its image. According to~\cite{H}, each hypersurface of degree $ d $ in $ \mathbf{P}^n $ is a hyperplane section of $ V_{d} \subset \mathbf{P}^N $ and viceversa.  
\end{rem}

For this reason we are allowed to associate to $ \mathcal{D} = \{D_{1}, \ldots, D_{\ell}\} $ a hyperplane arrangement $ \mathcal{H} = \{H_{1}, \ldots, H_{\ell}\} $ on $ \mathbf{P}^N $. In particular, let us assume that $ \mathcal{H} $ has normal crossings and let us denote by $ \Omega_{\mathbf{P}^N}^{1}(\log \mathcal{H}) $ the logarithmic bundle attached to it. As we can see in the proof of Theorem \ref{t54} (exact sequence (\ref{eq:prop211Dolg})), the vector bundles $ \Omega_{\mathbf{P}^n}^{1}(\log \mathcal{D}) $ and $ \Omega_{\mathbf{P}^N}^{1}(\log \mathcal{H}) $ are strictly related one to the other. 

Given the logarithmic bundle $ \Omega_{\mathbf{P}^n}^{1}(\log \mathcal{D}) $, the key idea is to reconstruct the hypersurfaces in $ \mathcal{D} $ as \emph{unstable hypersurfaces} of degree $ d $ of $ \Omega_{\mathbf{P}^n}^{1}(\log \mathcal{D}) $, using the fact that we are able to deal with hyperplanes. The notion of unstable hypersurface that we introduce in the following is very close to the one of unstable hyperplane recalled in (\ref{eq:unsthyp}).

\begin{defin}\label{d52}
Let $ D \subset \mathbf{P}^n $ be a hypersurface of degree $ d $. We say that $ D $ is \emph{unstable} for $ \Omega_{\mathbf{P}^n}^{1}(\log \mathcal{D}) $ if the following condition holds:
\begin{equation}\label{eq:ipinstabile}
H^{0}(D, {\Omega_{\mathbf{P}^n}^{1}(\log \mathcal{D})}^{\vee}_{|_{D}}) \not= \{0\}.
\end{equation}
\end{defin}

\begin{rem}\label{r53}
The previous definition is meaningful if 
$$ \ell > \displaystyle{{n+1} \over {d}}. $$
Indeed, in this case, from the short exact sequence given in Theorem \ref{T:Ancona} for $ \Omega_{\mathbf{P}^n}^{1}(\log \mathcal{D}) $
\begin{equation}\label{eq:Anconapermolteconiche}
0 \longrightarrow \mathcal{O}_{\mathbf{P}^n}(-d)^{\ell} \longrightarrow \mathcal{O}_{\mathbf{P}^n}(-1)^{n+1} \oplus \mathcal{O}_{\mathbf{P}^n}^{\ell -1} \longrightarrow \Omega_{\mathbf{P}^n}^{1}(\log \mathcal{D})  \longrightarrow 0
\end{equation}
it follows that $ c_{1}(\Omega_{\mathbf{P}^n}^{1}(\log \mathcal{D})^{\vee}) < 0 $ and so, by using Bohnhorst-Spindler criterion $($\cite{BS}$)$, $ \Omega_{\mathbf{P}^n}^{1}(\log \mathcal{D})^{\vee} $ is stable in the sense of \emph{Mumford-Takemoto}. These two facts imply that $ h^{0}(\mathbf{P}^n, \Omega_{\mathbf{P}^n}^{1}(\log \mathcal{D})^{\vee}) = 0 $.  
\end{rem}

We have the following:

\begin{lemma}\label{l:inst}
Let $ \mathcal{D} = \{D_{1}, \ldots, D_{\ell}\} $ be an arrangement of smooth hypersurfaces with normal crossings on $ \mathbf{P}^n $. Then $ D_{j} $ is unstable for $ \Omega_{\mathbf{P}^n}^{1}(\log \mathcal{D}) $ for all $ j \in \{1, \ldots, \ell\} $.
\end{lemma}
\begin{proof}
Let us consider the \emph{residue exact sequence} (\ref{eq:resexseq}) for $ \Omega_{\mathbf{P}^n}^{1}(\log \mathcal{D}) $, that is
$$ 0 \longrightarrow \Omega_{\mathbf{P}^n}^{1} \longrightarrow \Omega_{\mathbf{P}^n}^{1}(\log \mathcal{D}) \buildrel \rm res \over\longrightarrow \displaystyle\bigoplus_{i=1}^\ell \mathcal{O}_{D_{i}}  \longrightarrow 0. $$
If we restrict it to $ D_{j} $, we get a surjective map 
$$ \Omega_{\mathbf{P}^n}^{1}(\log \mathcal{D})_{|_{D_{j}}} \longrightarrow \mathcal{O}_{D_{j}} \oplus \displaystyle\bigoplus_{i=1, i \not= j}^\ell \mathcal{O}_{D_{i} \cap D_{j}} $$
from which we obtain a non zero map 
$$ \Omega_{\mathbf{P}^n}^{1}(\log \mathcal{D})_{|_{D_{j}}} \longrightarrow \mathcal{O}_{D_{j}}. $$
Thus 
$$ H^{0}(D_{j}, {\Omega_{\mathbf{P}^n}^{1}(\log \mathcal{D})}^{\vee}_{|_{D_{j}}}) = Hom(\mathcal{O}_{D_{j}}, {\Omega_{\mathbf{P}^n}^{1}(\log \mathcal{D})}^{\vee}_{|_{D_{j}}}) \not= \{0\} $$
that is $ D_{j} $ satisfies (\ref{eq:ipinstabile}). 
\end{proof}

Now we can state and prove the main result concerning the \emph{Torelli problem} in the case of arrangements with a \emph{large} number of hypersurfaces of the same degree.

\begin{thm0}\label{t54}
Let $ \mathcal{D} = \{D_{1}, \ldots, D_{\ell}\} $ be an arrangement of smooth hypersurfaces of degree $ d \geq 2 $ with normal crossings on $ \mathbf{P}^n $, with $ n \geq 2 $. Let $ \mathcal{H} = \{H_{1}, \ldots, H_{\ell}\} $ be the corresponding hyperplane arrangement on $ \mathbf{P}^N $ in the sense of Remark \ref{r:Veronese}. Assume that:
\begin{itemize}
\item[$1)$] $ \ell \geq N+ 4 $;
\item[$2)$] $ \mathcal{H} $ is a hyperplane arrangement with normal crossings;
\item[$3)$] $ H_{1}, \ldots, H_{\ell} $ don't osculate a rational normal curve of degree $ N $ in $ \mathbf{P}^N $.
\end{itemize}
Then $ \mathcal{D} $ is equal to the following set:
$$ \{ D \subset \mathbf{P}^n \, \emph{smooth irreducible hypersurface of degree d} \, | \, D \, \emph{satisfies} ~(\ref{eq:ipinstabile})\}. $$
\end{thm0}

\begin{proof}
The first inclusion is a direct consequence of Lemma \ref{l:inst}. \\
So, let us assume that $ D \subset \mathbf{P}^n $ is a smooth irreducible hypersurface of degree $ d $ which is unstable for $ \Omega_{\mathbf{P}^n}^{1}(\log \mathcal{D}) $, we want to prove that $ D \in \mathcal{D} $. It suffices to show that the hyperplane $ H \subset \mathbf{P}^N $ associated to $ D $ by means of $ \nu_{d} $ is unstable for $ \Omega_{\mathbf{P}^N}^{1}(\log \mathcal{H}) $: namely, if this is the case, since hypothesis $ 1),\, 2), \, 3) $ hold, Theorem \ref{T:V0} assures us that $ H \in \mathcal{H} $, that is $ H = H_{i} $ for $ i \in \{1, \ldots, \ell\} $ and so $ D = D_{i} \in \mathcal{D} $. \\
Since $ V_{d} $ is a non singular subvariety of $ \mathbf{P}^N $ which, by construction, intersects transversally $ \mathcal{H} $, from Proposition $ 2.11 $ of~\cite{Do} we get the following exact sequence:
\begin{equation}\label{eq:prop211Dolg}
0 \longrightarrow  \mathcal{N}_{V_{d}, \, \mathbf{P}^N}^{\vee} \longrightarrow \Omega_{\mathbf{P}^N}^{1}(\log \mathcal{H})_{|_{V_{d}} } \longrightarrow \Omega_{V_{d}}^{1}(\log \mathcal{H} \cap V_{d}) \longrightarrow 0 
\end{equation}
where $ \mathcal{N}_{V_{d}, \, \mathbf{P}^N}^{\vee} $ denotes the conormal sheaf of $ V_{d} $ in $ \mathbf{P}^N $. \\
We remark that $ V_{d} \cong \mathbf{P}^n $ and $ \mathcal{D} = \mathcal{H} \cap V_{d} $, so (\ref{eq:prop211Dolg}) becomes
\begin{equation}\label{eq:prop211Dolg2}
0 \longrightarrow  \mathcal{N}_{V_{d}, \, \mathbf{P}^N}^{\vee} \longrightarrow \Omega_{\mathbf{P}^N}^{1}(\log \mathcal{H})_{|_{\mathbf{P}^n}} \longrightarrow \Omega_{\mathbf{P}^n}^{1}(\log \mathcal{D}) \longrightarrow 0. 
\end{equation}
Restricting (\ref{eq:prop211Dolg2}) to $ D $ and then applying $ \mathcal{H}om(\cdot, \, \mathcal{O}_{D}) $ we obtain the following short exact sequence:
\begin{equation}
0 \longrightarrow {\Omega_{\mathbf{P}^n}^{1}(\log \mathcal{D})}^{\vee}_{|_{D}} \longrightarrow {\Omega_{\mathbf{P}^N}^{1}(\log \mathcal{H})}^{\vee}_{|_{D}} \longrightarrow ({\mathcal{N}_{V_{d}, \, \mathbf{P}^N\,_{|_{D}}}^{\vee}})^{\vee} \longrightarrow 0. 
\end{equation}
Finally, passing to cohomology we get
$$ 0 \longrightarrow H^{0}(D, {\Omega_{\mathbf{P}^n}^{1}(\log \mathcal{D})}^{\vee}_{|_{D}}) \longrightarrow H^{0}(D, {\Omega_{\mathbf{P}^N}^{1}(\log \mathcal{H})}^{\vee}_{|_{D}}). $$
By assumption, $ D $ is unstable for $ \Omega_{\mathbf{P}^n}^{1}(\log \mathcal{D}) $, that is condition (\ref{eq:ipinstabile}) holds. Necessarily it has to be 
\begin{equation}\label{eq:keyfact1}
H^{0}(D, {\Omega_{\mathbf{P}^N}^{1}(\log \mathcal{H})}^{\vee}_{|_{D}}) \not= 0.
\end{equation}
Now, let $ \mathcal{I}_{V_{d}, \, \mathbf{P}^N} $ be the ideal sheaf of $ V_{d} $ in $ \mathbf{P}^N $; we have the exact sequence
\begin{equation}\label{eq:idealsheafsequence}
0 \longrightarrow \mathcal{I}_{V_{d}, \, \mathbf{P}^N} \longrightarrow \mathcal{O}_{\mathbf{P}^N} \longrightarrow \mathcal{O}_{V_{d}} \longrightarrow 0.
\end{equation}
Since $ V_{d} \not\subset H $ we have
\begin{equation}\label{eq:idealsheafsequenceres}
0 \longrightarrow \mathcal{I}_{V_{d} \cap H, \, H} \longrightarrow \mathcal{O}_{H} \longrightarrow \mathcal{O}_{D} \longrightarrow 0
\end{equation}
By tensor product with $ {\Omega_{\mathbf{P}^N}^{1}(\log \mathcal{H})}^{\vee} $, (\ref{eq:idealsheafsequenceres}) becomes
$$ 0 \longrightarrow \mathcal{I}_{V_{d} \cap H, \, H} \otimes {\Omega_{\mathbf{P}^N}^{1}(\log \mathcal{H})}^{\vee}_{|_{H}} \longrightarrow {\Omega_{\mathbf{P}^N}^{1}(\log \mathcal{H})}^{\vee}_{|_{H}} \longrightarrow {\Omega_{\mathbf{P}^N}^{1}(\log \mathcal{H})}^{\vee}_{|_{D}} \longrightarrow 0. $$
Passing to cohomology we get
$$ 0 \longrightarrow H^{0}(H, \mathcal{I}_{V_{d} \cap H, \, H} \otimes {\Omega_{\mathbf{P}^N}^{1}(\log \mathcal{H})}^{\vee}_{|_{H}}) \longrightarrow H^{0}(H, {\Omega_{\mathbf{P}^N}^{1}(\log \mathcal{H})}^{\vee}_{|_{H}}) \longrightarrow $$
$$ \,\,\,\,\,\, \longrightarrow H^{0}(D, {\Omega_{\mathbf{P}^N}^{1}(\log \mathcal{H})}^{\vee}_{|_{D}}) \longrightarrow H^{1}(H, \mathcal{I}_{V_{d} \cap H, \, H} \otimes {\Omega_{\mathbf{P}^N}^{1}(\log \mathcal{H})}^{\vee}_{|_{H}}). \quad\quad  $$
To conclude the proof it suffices to show that 
\begin{equation}\label{eq:keyfact2}
H^{1}(H, \mathcal{I}_{V_{d} \cap H, \, H} \otimes {\Omega_{\mathbf{P}^N}^{1}(\log \mathcal{H})}^{\vee}_{|_{H}}) = \{0\}.
\end{equation}
In order to prove (\ref{eq:keyfact2}), we remark that, since $ \ell \geq N+4 $ (hypothesis 1)) and $ \mathcal{H} $ has normal crossings (hypothesis 2)), $ \Omega_{\mathbf{P}^N}^{1}(\log \mathcal{H}) $ is a \emph{Steiner} bundle over $ \mathbf{P}^N $, i.e.
$$ 0 \longrightarrow \mathcal{O}_{\mathbf{P}^N}(-1)^{\ell-N-1} \longrightarrow \mathcal{O}_{\mathbf{P}^N}^{\ell-1} \longrightarrow \Omega_{\mathbf{P}^N}^{1}(\log \mathcal{H}) \longrightarrow 0 $$
is exact. Since in the previous sequence all the terms are vector bundles, applying $ \mathcal{H}om(\cdot, \, \mathcal{O}_{\mathbf{P}^N}) $ we get
$$ 0 \longrightarrow {\Omega_{\mathbf{P}^N}^{1}(\log \mathcal{H})}^{\vee} \longrightarrow \mathcal{O}_{\mathbf{P}^N}^{\ell-1} \longrightarrow \mathcal{O}_{\mathbf{P}^N}(1)^{\ell-N-1} \longrightarrow 0, $$
which, via tensor product with $ \mathcal{I}_{V_{d},\,\mathbf{P}^N\,|_{H}} $, becomes
$$ 0 \longrightarrow \mathcal{I}_{V_{d} \cap H, \, H} \otimes \, {\Omega_{\mathbf{P}^N}^{1}(\log \mathcal{H})}^{\vee}_{|_{H}} \longrightarrow \mathcal{I}_{V_{d} \cap H, \, H} \otimes \, \mathcal{O}_{\mathbf{P}^N \, |_{H}}^{\ell-1} \longrightarrow $$
$$ \longrightarrow \mathcal{I}_{V_{d} \cap H, \, H} \otimes \, \mathcal{O}_{\mathbf{P}^N}(1)^{\ell-N-1}_{|_{H}} \longrightarrow 0. \quad\quad\quad\quad\quad\quad\quad\quad\,\,\, $$ 
Passing to cohomology we obtain
\begin{equation}\label{eq:coomolSteiner1}
\ldots \longrightarrow H^{0}(H, \mathcal{I}_{V_{d} \cap H, \, H} \otimes \, \mathcal{O}_{\mathbf{P}^N}(1)^{\ell-N-1}_{|_{H}}) \longrightarrow \quad\quad\quad\quad\quad\quad
\end{equation}
$$ \quad\quad\quad\,\,\, \longrightarrow H^{1}(H, \mathcal{I}_{V_{d} \cap H, \, H} \otimes \, {\Omega_{\mathbf{P}^N}^{1}(\log \mathcal{H})}^{\vee}_{|_{H}}) \longrightarrow H^{1}(H, \mathcal{I}_{V_{d} \cap H, \, H} \otimes \, \mathcal{O}_{\mathbf{P}^N \, |_{H}}^{\ell-1}). \quad \,$$
We remark that
$$ H^{i}(H, \mathcal{I}_{V_{d} \cap H, \, H} \otimes \, \mathcal{O}_{\mathbf{P}^N}(t)^{s}_{|_{H}}) = H^{i}(H, \mathcal{I}_{V_{d} \cap H, \, H}(t))^{\oplus s} $$
for all $ i, s, t $ integers such that $ i, s \geq 0 $. We note also that $ H^{0}(H, \mathcal{I}_{V_{d} \cap H, \, H}(1)) $ is the set of all homogeneous forms of degree $ 1 $ over $ H $ vanishing at $ V_{d} \cap H $ and so it is equal to $\{0\}$ . Thus (\ref{eq:coomolSteiner1}) reduces to 
$$ 0 \longrightarrow H^{1}(H, \mathcal{I}_{V_{d} \cap H, \, H} \otimes \, {\Omega_{\mathbf{P}^N}^{1}(\log \mathcal{H})}^{\vee}_{|_{H}}) \longrightarrow H^{1}(H, \mathcal{I}_{V_{d} \cap H, \, H})^{\oplus \ell-1}. $$
If we consider the induced cohomology sequence of (\ref{eq:idealsheafsequenceres}) we get that 
$$ H^{1}(H, \mathcal{I}_{V_{d} \cap H, \, H}) = {\mathbf{C}}^{k-1} $$
where $ k $ denotes the number of connected components of $ V_{d} \cap H $. Since $ V_{d} \cap H $ is connected, $ k = 1 $ and so (\ref{eq:keyfact2}) holds.
\end{proof}
Since isomorphic logarithmic bundles have the same set of unstable hypersurfaces, we have the following:

\begin{cor}
If $ \ell \geq N+4 $ then the map in (\ref{eq:Torellimap}) is generically injective.
\end{cor}

\begin{rem}
We don't know if Theorem \ref{t54} holds also without hypothesis 3).   
\end{rem}

\begin{rem}
In the case of arrangements of smooth quadrics with normal crossings on $ \mathbf{P}^n $, hypothesis 1) of Theorem \ref{t54} becomes $ \ell \geq {{(n+1)(n+2)} \over {2}} + 3 $, which translates in $ \ell \geq 9 $ if $ n=2 $. In the next two sections we will describe the cases of $ \ell = 1 $ and $ \ell = 2 $.
\end{rem}

\section{One quadric}

Arrangements consisting of one smooth quadric are not of \emph{Torelli type}. In this sense we have the following:

\begin{prop}\label{p:1quadrica}
Let $ Q \subset \mathbf{P}^n $ be a smooth quadric and let $ \mathcal{D} = \{Q\} $. Then 
\begin{equation}\label{eq:isom1quad}
\Omega_{\mathbf{P}^n}^{1}(\log \mathcal{D}) \cong \mathbf{TP}^n(-2).
\end{equation} 
\end{prop}

\begin{proof}
Let  us consider the short exact sequence for $ \Omega_{\mathbf{P}^n}^{1}(\log \mathcal{D}) $:
\begin{equation}\label{eq:Anconaper1}
0 \longrightarrow \mathcal{O}_{\mathbf{P}^n}(-2) \buildrel \rm M \over \longrightarrow \mathcal{O}_{\mathbf{P}^n}(-1)^{n+1} \longrightarrow \Omega_{\mathbf{P}^n}^{1}(\log \mathcal{D})  \longrightarrow 0.
\end{equation} 
where $ M $ is the matrix associated to the injective map defined by the three partial derivatives of a quadratic polynomial defining $ Q $. Without loss of generality we can assume that 
$$ M = \begin{pmatrix} x_{0} \cr \vdots \cr x_{n} \cr \end{pmatrix} $$
and so, by tensor product with $ \mathcal{O}_{\mathbf{P}^n}(1) $, (\ref{eq:Anconaper1}) becomes the Euler sequence for $ \mathbf{TP}^n(-1) $, which concludes the proof.                                                                               
\end{proof}

\begin{rem}
The previous result confirms Theorem \ref{T:UY2} for $ d=2 $ and yields a description of the logarithmic bundle in this case. Indeed, a quadric is defined by an equation which is always of \emph{Sebastiani-Thom type}. Moreover, a direct consequence of Proposition \ref{p:1quadrica} and Proposition \ref{p:n+2iperpiani} is that, if $ \mathcal{H} $ is an arrangement made of $ n+2 $ hyperplanes with normal crossings on $ \mathbf{P}^n $ and $ \mathcal{D} $ is as above, then $ \Omega_{\mathbf{P}^n}^{1}(\log \mathcal{D}) \cong \Omega_{\mathbf{P}^n}^{1}(\log \mathcal{H})(-1) $.
\end{rem}

\section{Pairs of quadrics}

Let's start with a characterization of pairs of quadrics with normal crossings. For the detailed proof see \cite{O2} (Theorem $ 8.2 $) or \cite{A} (Theorem 7.3).

\begin{thm0}\label{T:coppiequadriche}
Let $ Q_{1} $ and $ Q_{2} $ be smooth quadrics in $ \mathbf{P}^n $. \\
The following facts are equivalent:
\begin{itemize}
\item[$1)$] $ \mathcal{D} = \{ Q_{1},Q_{2}\} $ is an arrangement with normal crossings in $ \mathbf{P}^n $, that is $ Q_{1} \cap Q_{2} $ is a smooth codimension two subvariety;
\item[$2)$] in the pencil of quadrics generated by $ Q_{1} $ and $ Q_{2} $ there are $ n+1 $ distinct singular quadrics $($cones$)$ with singular points $ \{v_{0}, \ldots , v_{n}\} $.
\end{itemize}
\end{thm0}

\begin{rem}
Let $ C_{1} $ and $ C_{2} $ be smooth conics in $ \mathbf{P}^2 $. From the B\'ezout's Theorem it follows that the condition of normal crossings is equivalent to the fact that the pencil of conics generated by $ C_{1} $ and $ C_{2} $ has four distinct base points, which we denote by $ \{P,Q,R,S\} $. In \cite{A} we give a proof of Theorem \ref{T:coppiequadriche} in the case of $ n = 2 $ by using the above stated equivalence. In particular, the three singular conics in the pencil of $ C_{1} $ and $ C_{2} $ are three pairs of lines, with singular points denoted by $ \{E,F,G\} $ $($see figure $ 1 $$)$.
\end{rem}

\begin{figure}[h]
    \centering
\includegraphics[width=60mm]{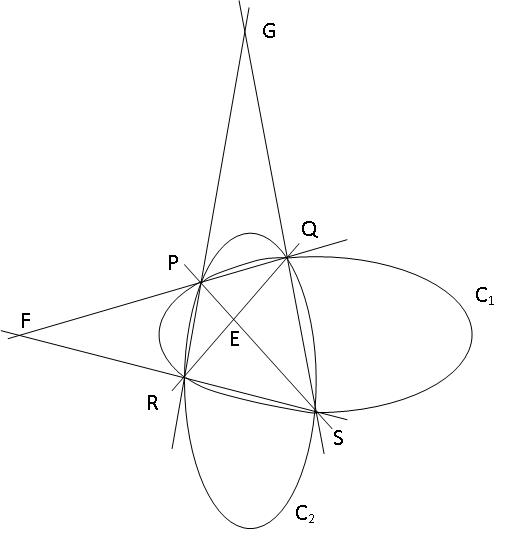}
    \caption{Two conics with normal crossings}
\label{flattening}
    \end{figure}

Now let's come back to the \emph{Torelli problem}. \\
Let $ \Omega_{\mathbf{P}^n}^{1}(\log \mathcal{D}) $ be the logarithmic bundle attached to an arrangement of smooth quadrics $ \mathcal{D} = \{Q_{1}, Q_{2}\} $ with normal crossings in $ \mathbf{P}^n $. Theorem \ref{T:Ancona} asserts that $ \Omega_{\mathbf{P}^n}^{1}(\log \mathcal{D}) $ is a rank $ n $ vector bundle over $ \mathbf{P}^n $ with the following short exact sequence:
\begin{equation}\label{eq:Ancona2quadriche} 
0 \longrightarrow \mathcal{O}_{\mathbf{P}^n}(-2)^{2} \longrightarrow \mathcal{O}_{\mathbf{P}^n}(-1)^{n+1} \oplus \mathcal{O}_{\mathbf{P}^n} \longrightarrow \Omega_{\mathbf{P}^n}^{1}(\log \mathcal{D})  \longrightarrow 0. 
\end{equation}
So, Bohnhorst-Spindler criterion (\cite{BS}) implies that $ \Omega_{\mathbf{P}^n}^{1}(\log \mathcal{D}) $ is a stable bundle. At this point, in the case of $ n=2 $ we can say that $ \mathcal{D} $ is not an arrangement of \emph{Torelli type}: indeed, the normalized bundle of $ \Omega_{\mathbf{P}^2}^{1}(\log \mathcal{D}) $ belongs to the moduli space $ \mathbf{M}_{\mathbf{P}^2}(-1,3) $ of stable rank-$2$ vector bundles on $ \mathbf{P}^2 $ with Chern classes $ -1 $ and $ 3 $, which, as we can see in \cite{OSS}, satisfies
\begin{equation}\label{eq:dimmodspace}
dim \mathbf{M}_{\mathbf{P}^2}(-1,3) = 8,
\end{equation}
while the number of parameters identifying a pair of conics is $10$. \\
Coming back to the general case, by using (\ref{eq:Ancona2quadriche}) we get that 
$$ H^{0}(\mathbf{P}^n,\Omega_{\mathbf{P}^n}^{1}(\log \mathcal{D})) = \mathbf{C} $$ 
and the $n$-th Chern class of $ \Omega_{\mathbf{P}^n}^{1}(\log \mathcal{D}) $ is equal to $ n+1 $. In the following we prove that $ \Omega_{\mathbf{P}^n}^{1}(\log \mathcal{D}) $ has one non-zero section with $ n+1 $ zeroes:  

\begin{prop}\label{p:zeroes}
Let $ \mathcal{D} = \{Q_{1},Q_{2}\} $ be an arrangement of smooth quadrics in $ \mathbf{P}^{n} $ with normal crossings and let $ \{v_{0}, \ldots, v_{n}\} $ as in Theorem \ref{T:coppiequadriche}. Then $ \{v_{0}, \ldots, v_{n}\} $ is the zero locus of the non-zero section of $ \Omega_{\mathbf{P}^n}^{1}(\log \mathcal{D}) $.
\end{prop}

\begin{proof}
Assume that $ A=\{a_{ij}\} $ and $ B=\{b_{ij}\} $ are the matrices representing $ Q_{1} $ and $ Q_{2} $ with respect to the canonical basis of $ \mathbf{C}^{n+1} $ and let
$$ M = \begin{pmatrix} 
2 $$ \sum^n_{i=0}$$ a_{0i}x_{i} & 2 $$\sum^n_{i=0}$$  b_{0i}x_{i} \cr 
\vdots  &  \vdots  \cr 
2 $$\sum^n_{i=0}$$ a_{ni}x_{i}  &  2 $$\sum^n_{i=0}$$ b_{ni}x_{i}  \cr 
$$\sum^n_{i,j=0} $$ a_{ij}x_{i}x_{j} & 0 \cr
\end{pmatrix} $$
be the $ (n+2) \times 2 $ matrix associated to the exact sequence (\ref{eq:Ancona2quadriche}). \\
In order to determine the zeroes of the section of $ \Omega_{\mathbf{P}^n}^{1}(\log \mathcal{D}) $, we have to find $ x = (x_{0}, \ldots ,x_{n}) \in \mathbf{C}^{n+1} -\{\underline{0}\} $ such that the linear part of $ M $ has rank $ 1 $, that is the solutions of 
$$ Ax = \lambda Bx $$
for certain $ \lambda \in \mathbf{C} $ ($ \lambda $ is an eigenvalue of $ AB^{-1} $ and $ x $ is the corresponding eigenvector). In other words, any such $ x $ has to be a representative vector for the singular point of the quadric associated to $ A - \lambda B $, which concludes the proof.
\end{proof}

\begin{rem}
If $ n=2 $, in \cite{A} we prove also that the three lines through any two of the points in $ \{E,F,G\} $ $($see figure $3$$)$ are exactly the \emph{jumping lines} of the normalized bundle of $ \Omega_{\mathbf{P}^2}^{1}(\log \mathcal{D}) $.
\end{rem}

In order to state and prove the main result concerning pairs of quadrics in the complex projective space, we recall some preliminaries. If $ Q \subset \mathbf{P}^{n} $ is a smooth quadric, then $ Q^{\vee}\subset (\mathbf{P}^{n})^{\vee} $ is the \emph{dual quadric} of $ Q $, which is given by the tangent hyperplanes to $ Q $. In particular, if $ Q $ is represented by a symmetric $ n \times n $ matrix $ G $, then $ Q^{\vee} $ is associated to $ G^{-1} $. The set of tangent hyperplanes to two smooth quadrics with normal crossings in $ \mathbf{P}^{n} $, $ Q_{1} $ and $ Q_{2} $, is the base locus of the pencil of quadrics in $ (\mathbf{P}^{n})^{\vee} $ generated by $ Q_{1}^{\vee} $ and $ Q_{2}^{\vee} $, that is $ Q_{1}^{\vee} \cap Q_{2}^{\vee} $. \\
We have the following:

\begin{thm0}\label{T:pairquad}
Let $ \mathcal{D}_{1} = \{Q_{1},Q_{2}\} $ and $ \mathcal{D}_{2} = \{Q'_{1},Q'_{2}\} $ be arrangements of smooth quadrics with normal crossings in $ \mathbf{P}^{n} $. Then 
$$ \Omega_{\mathbf{P}^n}^{1}(\log \mathcal{D}_{1}) \cong \Omega_{\mathbf{P}^n}^{1}(\log \mathcal{D}_{2}) \Longleftrightarrow Q_{1}^{\vee} \cap Q_{2}^{\vee} = {Q_{1}'}^{\vee} \cap {Q'_{2}}^{\vee}. $$
\end{thm0}

\begin{proof}
Suppose that $\Omega_{\mathbf{P}^n}^{1}(\log \mathcal{D}_{1}) \cong \Omega_{\mathbf{P}^n}^{1}(\log \mathcal{D}_{2})$, Proposition \ref{p:zeroes} and Theorem \ref{T:coppiequadriche} imply that in a frame of $ \mathbf{C}^{n+1} $ given by representative vectors of the points $ \{v_{0}, \ldots, v_{n}\} $, the quadrics $ Q_{1} $, $ Q_{2} $, $ Q'_{1} $, $ Q'_{2} $ have equations, respectively:
$$ a_{0}x_{0}^2+a_{1}x_{1}^2+ \ldots + a_{n-1}x_{n-1}^2-x_{n}^2 = 0 $$
$$ b_{0}x_{0}^2+b_{1}x_{1}^2+ \ldots + b_{n-1}x_{n-1}^2-x_{n}^2 = 0 $$
$$ c_{0}x_{0}^2+c_{1}x_{1}^2+ \ldots + c_{n-1}x_{n-1}^2-x_{n}^2 = 0 $$
$$ d_{0}x_{0}^2+d_{1}x_{1}^2+ \ldots + d_{n-1}x_{n-1}^2-x_{n}^2 = 0 $$
where $ a_{i},b_{i}, c_{i}, d_{i} \in \mathbf{C} - \{0\} $, $ a_{i} \not= b_{i} $, $ c_{i} \not= d_{i} $, $ \displaystyle{{a_{i}} \over {a_{j}}} \not= {{b_{i}} \over {b_{j}}} $, $ \displaystyle{{c_{i}} \over {c_{j}}} \not= {{d_{i}} \over {d_{j}}} $, for all $ i,j \in \{0, \ldots, n-1 \} $ (we remark that our quadrics are smooth and in the pencil generated by them there are $ n+1 $ singular quadrics). Saying that the two logarithmic bundles are isomorphic is equivalent to the fact that we can find two invertible matrices 
\begin{equation}\label{eq:M'quadrics}
M' = \begin{pmatrix} 
\alpha & \beta \cr 
\gamma & \delta \cr
\end{pmatrix}
\end{equation}
\begin{equation}\label{eq:M''quadrics}
M'' = \begin{pmatrix}
E_{1,1} & \ldots & E_{1,n+1} & f_{1} \cr
E_{2,1} & \ldots & E_{2,n+1} & f_{2} \cr
\vdots & {} & \vdots & \vdots \cr
E_{n+1,1} & \ldots & E_{n+1,n+1} & f_{n+1} \cr
0 & \ldots & 0 & \theta \cr
\end{pmatrix}
\end{equation}
with $ \alpha, \beta, \gamma, \delta, E_{i,j}, \theta \in \mathbf{C} $ and $ f_{j} = \displaystyle \sum_{j = 0}^{n} f^{i}_{j}x_{i} $ complex linear forms, such that the diagram
$$ \mathcal{O}_{\mathbf{P}^n}(1)^{n+1} \oplus \mathcal{O}_{\mathbf{P}^n} \buildrel \rm N_{1} \over \longrightarrow \mathcal{O}_{\mathbf{P}^n}(2)^{2} $$ 
$$ \,\,\,\,\,\,\,\,\,\,M'' \downarrow  \,\,\,\,\,\,\,\,\,\,\,\,\,\,\,\,\,\,\,\,\,\,\,\,\,\,\,\,\,\ \downarrow M'  $$
$$ \mathcal{O}_{\mathbf{P}^n}(1)^{n+1} \oplus \mathcal{O}_{\mathbf{P}^n} \buildrel \rm N_{2} \over \longrightarrow \mathcal{O}_{\mathbf{P}^n}(2)^{2} $$ 
commutes. In this diagram $ N_{1} $ and $ N_{2} $ are the matrices associated to the two logarithmic bundles in the sense of Theorem \ref{T:Ancona}, that is
$$ N_{1} = \begin{pmatrix}
2a_{0}x_{0} & \dots  & 2a_{n-1}x_{n-1} & -2x_{n} &  a_{0}x_{0}^2+ \ldots + a_{n-1}x_{n-1}^2-x_{n}^2 \cr 
2b_{0}x_{0} & \ldots & 2b_{n-1}x_{n-1} & -2x_{n} & 0 \cr
\end{pmatrix} $$
$$ N_{2} = \begin{pmatrix}
2c_{0}x_{0} & \dots  & 2c_{n-1}x_{n-1} & -2x_{n} &  c_{0}x_{0}^2+ \ldots + c_{n-1}x_{n-1}^2-x_{n}^2 \cr 
2d_{0}x_{0} & \ldots & 2d_{n-1}x_{n-1} & -2x_{n} & 0 \cr
\end{pmatrix}. $$
Let's equate the entries of the $ 2 \times (n+2) $ matrices $ M'N_{1} $ and $ N_{2}M'' $, we get the following conditions:  
$$  E_{i,j} = 0 \,\,\,\,\,\,\,\, for\,\,all \,\,i,j \in \{1, \ldots ,n\}, \, i\not=j $$
\begin{equation}\label{eq:E_ii}
E_{i,i} = \displaystyle{{a_{i-1}} \over {c_{i-1}}} \alpha + {{b_{i-1}} \over {c_{i-1}}} \beta = \displaystyle{{a_{i-1}} \over {d_{i-1}}} \gamma + \displaystyle{{b_{i-1}} \over {d_{i-1}}} \delta \,\,\,\,\,\, for\,\, i \in \{1, \ldots ,n\}
\end{equation}
\begin{equation}\label{eq:E_n+1n+1}
E_{n+1,n+1} = \alpha + \beta = \gamma + \delta
\end{equation}
\begin{equation}\label{eq:1}
\alpha a_{i-1} = 2 c_{i-1} f^{i-1}_{i} + \theta c_{i-1}  \,\,\,\,\,\,\,\, for\,\,i \in \{1, \ldots ,n\}
\end{equation}
\begin{equation}\label{eq:2}
\alpha = 2f^{n}_{n+1} + \theta
\end{equation}
\begin{equation}\label{eq:3}
\gamma a_{i-1} = 2 d_{i-1} f^{i-1}_{i}   \,\,\,\,\,\, for\,\, i \in \{1, \ldots ,n\}
\end{equation}
\begin{equation}\label{eq:3bis}
\gamma = 2f^{n}_{n+1}
\end{equation}
\begin{equation}\label{eq:4}
c_{i-1}f^{i+j}_{i}+c_{i+j}f^{i-1}_{i+j+1} = 0    \,\,\,\,\,\, for \,\,i \in \{1, \ldots ,n-1\}, \, j \in \{i, \ldots, n-1\}
\end{equation}
\begin{equation}\label{eq:5}
c_{i-1}f^{n}_{i}-f^{i-1}_{n+1} = 0    \,\,\,\,\,\, for \,\,i \in \{1, \ldots ,n\}
\end{equation}
\begin{equation}\label{eq:6}
d_{i-1}f^{i+j}_{i}+d_{i+j}f^{i-1}_{i+j+1} = 0    \,\,\,\,\,\, for \,\,i \in \{1, \ldots ,n-1\}, \, j \in \{i, \ldots, n-1\}
\end{equation}
\begin{equation}\label{eq:7}
d_{i-1}f^{n}_{i}-f^{i-1}_{n+1} = 0    \,\,\,\,\,\, for \,\,i \in \{1, \ldots ,n\}.
\end{equation}
By using equations (\ref{eq:4}), (\ref{eq:5}), (\ref{eq:6}), (\ref{eq:7}) and remembering the properties of $ c_{0}, \ldots, c_{n-1}, d_{0}, \ldots, d_{n-1} $, we get that if $ j-i \not= 1 $ then $ f^{i}_{j}={0} $. So each linear form reduces to $ f_{j} = f^{j-1}_{j}x_{j -1} $. In order to determine these coefficients we consider equations (\ref{eq:1}), (\ref{eq:2}), (\ref{eq:3}), (\ref{eq:3bis}) (actually these are $ 2n+2 $ relations) and we get
\begin{equation}\label{eq:fi-1iintermedia1}
f^{i-1}_{i} =  \displaystyle {{\alpha} \over {2}} \left ({{a_{i-1}} \over {c_{i-1}}} - {{a_{i}} \over {c_{i}}} \right ) + f^{1}_{2}  \,\,\,\,\,\, for \,\,i \in \{1, \ldots ,n\}
\end{equation}
\begin{equation}\label{eq:fnn+1intermedia1}
f^{n}_{n+1} =  \displaystyle {{\alpha} \over {2}} \left (1 - {{a_{1}} \over {c_{1}}} \right ) + f^{1}_{2} 
\end{equation}
\begin{equation}\label{eq:thetaintermedia}
\theta =  \displaystyle {{a_{1}} \over {c_{1}}} \alpha - 2 f^{1}_{2}.
\end{equation}
\begin{equation}\label{eq:fi-1iintermedia2}
f^{i-1}_{i} =  \displaystyle {{a_{i-1}} \over {2d_{i-1}}} \gamma   \,\,\,\,\,\, for \,\,i \in \{1, \ldots ,n\}
\end{equation}
\begin{equation}\label{eq:fnn+1intermedia2}
f^{n}_{n+1} =  \displaystyle {{\gamma} \over {2}}. 
\end{equation}
If we consider (\ref{eq:fnn+1intermedia1}), (\ref{eq:fnn+1intermedia2}), (\ref{eq:fi-1iintermedia2}) for $  i=2 $, together with (\ref{eq:E_n+1n+1}) and (\ref{eq:E_ii}) for $ i=1 $, we get 
\begin{equation}\label{eq:beta}
\beta = \displaystyle {{a_{1}b_{0}c_{0}(c_{1}-d_{1})+a_{0}c_{1}d_{1}(d_{0}-c_{0})+a_{0}a_{1}(c_{0}d_{1}-c_{1}d_{0})} \over {b_{0}c_{1}(c_{0}-d_{0})(d_{1}-a_{1})}} \alpha
\end{equation}
\begin{equation}\label{eq:gamma}
\gamma = \displaystyle {{d_{1}(a_{1}-c_{1})} \over {c_{1}(a_{1}-d_{1})}} \alpha
\end{equation}
\begin{equation}\label{eq:delta}
\delta = \displaystyle {{a_{1}d_{1}(b_{0}d_{0}-a_{0}c_{0})+a_{1}c_{1}d_{0}(a_{0}-b_{0})-a_{0}c_{1}d_{1}(d_{0}-c_{0})} \over {b_{0}c_{1}(c_{0}-d_{0})(a_{1}-d_{1})}} \alpha.
\end{equation}
So, if we choose $ \alpha \in \mathbf{C} - \{0\} $ and the inequality
\begin{equation}\label{eq:M'quadricsinvertibile}
a_{1}(c_{1}d_{0}-c_{0}d_{1})+c_{1}d_{1}(c_{0}-d_{0}) \not= 0
\end{equation}
is satisfied, then the matrix $ M' $ introduced in (\ref{eq:M'quadrics}) is invertible. Moreover,  from (\ref{eq:gamma}), (\ref{eq:thetaintermedia}), (\ref{eq:fi-1iintermedia2}), (\ref{eq:fnn+1intermedia2}) we get that 
\begin{equation}\label{eq:fi-1i}
f^{i-1}_{i} = \displaystyle {{a_{i-1}d_{1}(a_{1}-c_{1})} \over {2d_{i-1}c_{1}(a_{1}-d_{1})}} \alpha \,\,\, for \,\, i \in \{1, \ldots, n\}
\end{equation}
$$ f^{n}_{n+1} = \displaystyle {{d_{1}(a_{1}-c_{1})} \over {2c_{1}(a_{1}-d_{1})}} \alpha $$
$$ \theta = \displaystyle {{a_{1}(c_{1}-d_{1})} \over {c_{1}(a_{1}-d_{1})}} \alpha. $$
If we consider (\ref{eq:fi-1iintermedia1}) and (\ref{eq:fi-1i}) we obtain $ n-1 $ resolubility conditions for our system involving the coefficients of the quadrics: for $ i \in \{1, \ldots ,n\} $
\begin{equation}\label{eq:first(n-1)rescond}
a_{1}a_{i-1}(c_{1}d_{i-1}-c_{i-1}d_{1})+a_{1}c_{i-1}d_{i-1}(d_{1}-c_{1})+a_{i-1}c_{1}d_{1}(c_{i-1}-d_{i-1}) = 0.  \end{equation}
Moreover, by using (\ref{eq:E_ii}) for $ i \in \{2, \ldots, n\} $ with (\ref{eq:beta}), (\ref{eq:gamma}), (\ref{eq:delta}) 
we get the following $ n-1 $ relations:
\begin{equation}\label{eq:conditionn}
a_{1}b_{1}(b_{0}-a_{0})(c_{0}d_{1}-c_{1}d_{0})+c_{1}d_{1}(c_{0}-d_{0})(a_{0}b_{1}-a_{1}b_{0}) = 0
\end{equation}
\begin{equation}\label{eq:condition2n-2}
a_{1}a_{i-1}b_{0}(c_{0}-d_{0})(c_{i-1}d_{1}-c_{1}d_{i-1})+a_{1}b_{0}b_{i-1}(d_{1}-c_{1})(c_{i-1}d_{0}-c_{0}d_{i-1})+
\end{equation}
$$ +a_{0}a_{1}b_{i-1}(c_{i-1}-d_{i-1})(c_{1}d_{0}-c_{0}d_{1})+c_{1}d_{1}(c_{i-1}-d_{i-1})(c_{0}-d_{0})(a_{0}b_{i-1}-a_{i-1}b_{0}) = 0  $$
for $ i \in \{1, \ldots ,n\} $. 
In particular, we can find final expressions for $ E_{i,i} $:
$$ E_{1,1} = \displaystyle {{a_{1}(c_{1}-d_{1})(a_{0}-b_{0})} \over {c_{1}(a_{1}-d_{1})(c_{0}-d_{0})}} \alpha $$
\leftline{$ E_{i,i} = $} 
$$ = \displaystyle {{c_{1}(a_{i}b_{0}-a_{0}b_{i-1})[d_{1}(d_{0}-c_{0})-a_{1}d_{0}]+a_{1}c_{0}[b_{0}c_{1}(a_{i-1}-b_{i-1})+b_{i-1}d_{1}(b_{0}-a_{0})]} \over {b_{0}c_{1}c_{i -1}(c_{0}-d_{0})(a_{1}-d_{1})}} \alpha $$
for all $ i \in \{2,3,\ldots,n \} $ and
$$ E_{n+1,n+1}= \displaystyle {{(a_{0}-b_{0})[a_{1}(c_{1}d_{0}-c_{0}d_{1})+c_{1}d_{1}(c_{0}-d_{0})]} \over {b_{0}c_{1}(c_{0}-d_{0})(a_{1}-d_{1})}} \alpha. $$
Thus the matrix $ M'' $ introduced in (\ref{eq:M''quadrics}) is invertible if and only if, fixed $ \alpha \in \mathbf{C} - \{0\} $, besides (\ref{eq:M'quadricsinvertibile}), for all $ i \in \{2,3,\ldots,n\} $ holds
\begin{equation}\label{eq:M''quadricsinvertibile}
c_{1}(a_{i-1}b_{0}-a_{0}b_{i-1})[d_{1}(d_{0}-c_{0})-a_{1}d_{0}]+a_{1}c_{0}[b_{0}c_{1}(a_{i-1}-b_{i-1})+b_{i-1}d_{1}(b_{0}-a_{0})] \not= 0.
\end{equation}
In this way we have that $ \Omega_{\mathbf{P}^n}^{1}(\log \mathcal{D}_{1}) \cong \Omega_{\mathbf{P}^n}^{1}(\log \mathcal{D}_{2}) $ if and only if the $ 2n-2 $ relations (\ref{eq:first(n-1)rescond}), (\ref{eq:conditionn}), (\ref{eq:condition2n-2}) and the $ n $ open conditions (\ref{eq:M'quadricsinvertibile}), (\ref{eq:M''quadricsinvertibile}) hold. If we fix the coefficients $ a_{0}, \ldots, a_{n-1}, b_{0}, \ldots, b_{n-1}, c_{0}, d_{0} $, the $ 2n-2 $ resolubility conditions imply that the matrix associated to $ Q'_{i} $, $ i \in \{1,2\} $, is of the form
\begin{equation}\label{eq:matriceQ'i}
\begin{pmatrix} 
\displaystyle{{a_{0}b_{0}(1+t_{i})} \over {b_{0}+t_{i}a_{0}}} & 0 & \ldots & \ldots & 0 \cr
0 & \displaystyle{{a_{1}b_{1}(1+t_{i})} \over {b_{1}+t_{i}a_{1}}} & 0 & \ldots & 0 \cr
\vdots & {} & \ddots & {} & \vdots \cr
0 & \ldots & 0 & \displaystyle{{a_{n-1}b_{n-1}(1+t_{i})} \over {b_{n-1}+t_{i}a_{n-1}}} & 0 \cr
0 & \ldots & \ldots & 0 & -1 \cr
\end{pmatrix}
\end{equation}
where $ t_{1} = \displaystyle {{b_{0}(a_{0}-c_{0})} \over {a_{0}(c_{0}-b_{0})}} $ and $ t_{2} = \displaystyle {{b_{0}(a_{0}-d_{0})} \over {a_{0}(d_{0}-b_{0})}} $. \\ Up to scalar multiplication, the previous matrix is equivalent to 
$$ (A^{-1}+t_{i}B^{-1})^{-1} $$
where $ A $ and $ B $ are the diagonal matrices associated to $ Q_{1} $ and $ Q_{2} $.  Since $ A^{-1} $ and $ B^{-1} $ represent the dual quadrics $ Q^{\vee}_{1} $ and $ Q^{\vee}_{2} $ in the dual projective space $ (\mathbf{P}^{n})^{\vee} $, then $ \mathcal{D}_{1} $ and $ \mathcal{D}_{2} $ have the same tangent hyperplanes, that is $ Q_{1}^{\vee} \cap Q_{2}^{\vee} = {Q_{1}'}^{\vee} \cap {Q'_{2}}^{\vee} $. We remark that this implication is true when the entries of the matrix in (\ref{eq:matriceQ'i}) satisfy the open condition (\ref{eq:M''quadricsinvertibile}) ((\ref{eq:M'quadricsinvertibile}) is always verified). \\
Viceversa, assume that $ \mathcal{D}_{1} $ and $ \mathcal{D}_{2} $ have the same tangent hyperplanes, we want to prove that they have isomorphic logarithmic bundles. Since $ Q_{1} $ and $ Q_{2} $ have normal crossings, Theorem \ref{T:coppiequadriche} allows us to suppose that they are represented by $ A = diag(a_{0}, \ldots, a_{n-1},-1) $ and $ B = diag(b_{0}, \ldots, b_{n-1},-1) $, as above. By hypothesis, $ {Q'_{1}}^{\vee} $ and $ {Q'_{2}}^{\vee} $ live in the pencil of quadrics generated by $ {Q_{1}}^{\vee} $ and $ {Q_{2}}^{\vee} $, that is $ Q'_{1} $ and $ Q'_{2} $ are represented by matrices like the one in (\ref{eq:matriceQ'i}). Clearly these matrices satisfy (\ref{eq:first(n-1)rescond}), (\ref{eq:conditionn}), (\ref{eq:M'quadricsinvertibile}). If also (\ref{eq:M''quadricsinvertibile}) holds, then $ \Omega_{\mathbf{P}^2}^{1}(\log \mathcal{D}_{1}) \cong \Omega_{\mathbf{P}^2}^{1}(\log \mathcal{D}_{2}) $, which concludes the proof.
\end{proof}

\begin{rem}
If $ \mathcal{D} $ is a pair of smooth conics with normal crossings, then Theorem \ref{T:pairquad} asserts that the isomorphism class of $ \Omega_{\mathbf{P}^2}^{1}(\log \mathcal{D}) $ is determined by the four tangent lines to $ \mathcal{D} $ (see figure $ 4 $). It is confirmed also by the dimensional computations that we recalled in (\ref{eq:dimmodspace}): indeed, the dimension of the moduli space $ \mathbf{M}_{\mathbf{P}^2}(-1,3) $ containing the normalized bundle of $ \Omega_{\mathbf{P}^2}^{1}(\log \mathcal{D}) $ equates the number of parameters determining four lines in $ \mathbf{P}^2 $. In particular, every element of $ \mathbf{M}_{\mathbf{P}^2}(-1,3) $ is a logarithmic bundle of a pair of smooth conics with normal crossings, twisted by $ -1 $.
\end{rem}

\begin{figure}[h]
    \centering
\includegraphics[width=50mm]{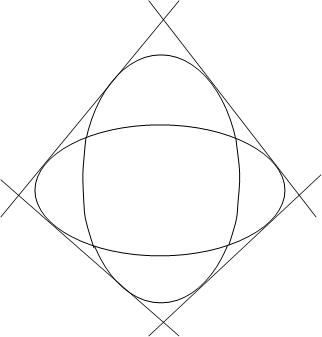}
    \caption{Four tangent lines of a pair of conics}
\label{flattening}
    \end{figure}

\vfill\eject

\end{document}